\documentclass[12pt]{amsart}

\usepackage[left=2.5cm, right=2.5cm, top=2.5cm, bottom=2.5cm]{geometry}

\usepackage{amssymb,amsmath,mathrsfs,amsthm,amscd}

\newcommand{\R}{\mathbb{R}}
\newcommand{\Rn}{{\R^n}}

\newcommand{\eps}{\varepsilon}

\newcommand{\cS}{{\mathscr{S}}}

\newcommand{\cQ}{{\mathcal{Q}}}
\newcommand{\cB}{{\mathcal{B}}}
\newcommand{\cR}{{\mathcal{R}}}
\newcommand{\cI}{{\mathcal{I}}}

\newcommand{\cC}{{\mathcal{C}}}

\newcommand{\cF}{{\mathcal{F}}}

\newcommand{\Lp}{{L^p}}

\newcommand{\Lone}{{L^1}}
\newcommand{\Loneloc}{{L^1_{\text{loc}}}}

\DeclareMathOperator*{\essinf}{ess\,inf}

\def\BMO#1#2{\mathrm{BMO}_{#1}^{#2}}
\def\BLO#1#2{\mathrm{BLO}_{#1}^{#2}}

\def\bBMO#1{\mathrm{\bf{BMO}}{}{#1}}

\def\norm#1#2{\lVert{#1}\rVert_{#2}}

\def\Xint#1{\mathchoice
{\XXint\displaystyle\textstyle{#1}}%
{\XXint\textstyle\scriptstyle{#1}}%
{\XXint\scriptstyle\scriptscriptstyle{#1}}%
{\XXint\scriptscriptstyle\scriptscriptstyle{#1}}%
\!\int}
\def\XXint#1#2#3{{\setbox0=\hbox{$#1{#2#3}{\int}$ }
\vcenter{\hbox{$#2#3$ }}\kern-.57\wd0}}

\def\dashint{\Xint-}

\newtheorem*{open}{Open Problem}
\newtheorem*{introtheorem}{Theorem}

\newtheorem{theorem}{Theorem}[section]

\newtheorem{proposition}[theorem]{Proposition}

\theoremstyle{definition}

\newtheorem{definition}[theorem]{Definition}
\newtheorem{example}[theorem]{Example}

\theoremstyle{remark}
\newtheorem{remark}[theorem]{Remark}

\numberwithin{equation}{section}

\begin{document}


\title[{\bf Geometric maximal operators and $\bBMO{}{}$ on product bases}]{\bf Geometric maximal operators and $\bBMO{}{}$ on product bases}
\date{}

\author[Dafni]{Galia Dafni}
\address{(G.D.) Concordia University, Department of Mathematics and Statistics, Montr\'{e}al, Qu\'{e}bec, H3G-1M8, Canada}
\curraddr{}
\email{galia.dafni@concordia.ca}
\thanks{G.D. was partially supported by the Natural Sciences and Engineering Research Council (NSERC) of Canada, the Centre de recherches math\'{e}matiques (CRM), and the Fonds de recherche du Qu\'{e}bec -- Nature et technologies (FRQNT). R.G. was partially supported by the Centre de recherches math\'{e}matiques (CRM), the Institut des sciences math\'{e}matiques (ISM), and the Fonds de recherche du Qu\'{e}bec -- Nature et technologies (FRQNT). H.Y. was partially supported by GCSU Faculty Development Funds. }

\author[Gibara]{Ryan Gibara}
\address{(R.G.) Concordia University, Department of Mathematics and Statistics, Montr\'{e}al, Qu\'{e}bec, H3G-1M8, Canada}
\curraddr{}
\email{ryan.gibara@concordia.ca}

\author[Yue]{Hong Yue}
\address{(H.Y.) Georgia College and State University, Department of Mathematics, Milledgeville, Georgia, 31061, USA}
\curraddr{}
\email{hong.yue@gcsu.edu}

\begin{abstract} 
We consider the problem of the boundedness of maximal operators on $\BMO{}{}$ on shapes in $\R^n$. We prove that for bases of shapes with an engulfing property, the corresponding maximal function is bounded from $\BMO{}{}$ to $\BLO{}{}$, generalising a known result of Bennett for the basis of cubes. When the basis of shapes does not possess an engulfing property but exhibits a product structure with respect to lower-dimensional shapes coming from bases that do possess an engulfing property, we show that the corresponding maximal function is bounded from $\BMO{}{}$ to a space we define and call rectangular $\BLO{}{}$. 
\end{abstract}

\maketitle


\section{{\bf Introduction}}


The uncentred Hardy-Littlewood maximal function, $Mf$, of a function $f\in \Loneloc(\R^n)$ is defined as
\begin{equation}
\label{maximal}
Mf(x)=\sup_{Q\ni{x}}\dashint_{Q}\!|f|=\sup_{Q\ni{x}}\frac{1}{|Q|}\int_{Q}\!|f|,
\end{equation}
where the supremum is taken over all cubes $Q$ containing the point $x$ and $|Q|$ is the measure of the cube. Note that, unless otherwise stated, cubes in this paper will mean cubes with sides parallel to the axes. The well-known Hardy-Littlewood-Wiener theorem states that the operator $M$ is bounded from $\Lp(\R^n)$ to $\Lp(\R^n)$ for $1<p\leq\infty$ and from $\Lone(\R^n)$ to $L^{1,\infty}(\R^n)$ (see Stein \cite{st}). 

This maximal function is a classical object of study in real analysis due to its connection with differentiation of the integral. When the cubes in \eqref{maximal} are replaced by rectangles (the Cartesian product of intervals), we have the strong maximal function, $M_s$, which is also bounded from $\Lp(\R^n)$ to $\Lp(\R^n)$ for $1<p\leq\infty$ but is not bounded from $\Lone(\R^n)$ to $L^{1,\infty}(\R^n)$. Its connection to what is known as strong differentiation of the integral is also quite classical (see Jessen-Marcinkiewicz-Zygmund \cite{jmz}).

When the cubes in \eqref{maximal} are replaced by more general sets taken from a basis $\cS$, we obtain a geometric maximal operator, $M_{\cS}$ (we follow the nomenclature of \cite{hs}). Here the subscript $\cS$ emphasizes that the behaviour of this operator depends on the geometry of the sets in $\cS$, which we call shapes. Such maximal operators have been extensively studied; see, for instance, the monograph of de Guzm\'{a}n (\cite{guz}). A key theme in this area is the identification of the weakest assumptions needed on $\cS$ to guarantee certain properties of $M_{\cS}$. For examples of the kind of research currently being done in this area, including its connection to the theory of $A_p$ weights, see \cite{glpt,hlp,hs,lpr,stok1,stok2}. 

Introduced by John and Nirenberg in \cite{jn} for functions supported on a cube, the space of functions of bounded mean oscillation, $\BMO{}{}(\R^n)$, is the set of all $f\in \Loneloc(\R^n)$ such that
\begin{equation}
\label{bmo}
\sup_{Q}\dashint_{Q}\!|f-f_Q|<\infty,
\end{equation}
where $f_Q=\dashint_{Q}f$ is the mean of $f$ over the cube $Q$ and the supremum is taken over all cubes $Q$. 

An important subset of $\BMO{}{}(\R^n)$, introduced by Coifman and Rochberg in \cite{cr}, is the class of functions of bounded lower oscillation, $\BLO{}{}(\R^n)$. The definition of this class is obtained by replacing the mean $f_Q$ in \eqref{bmo} by $\essinf\limits_{Q}f$, the essential infimum of $f$ on the cube $Q$.

Just as cubes can be replaced by rectangles in the definition of the maximal function, the same can be done with the definition of $\BMO{}{}(\R^n)$ (and, likewise, with $\BLO{}{}(\R^n)$). The resulting space, strong $\BMO{}{}$, has appeared in the literature under different names (see \cite{cs,dlowy,kor}). 

Pushing the analogy with maximal functions even further, one may replace the cubes in \eqref{bmo} by more general shapes, coming from a basis $\cS$. This space, $\BMO{\cS}{}(\R^n)$, was introduced in previous work of two of the authors in \cite{dg}. In this work, a product characterisation of $\BMO{\cS}{}(\R^n)$ was shown when the shapes in $\cS$ exhibit some product structure. 

In the two-parameter setting of $\R^n\times\R^m$, there is a related space, rectangular $\BMO{}{}$, that is larger than strong $\BMO{}{}$. The unacquainted reader is invited to see \cite{carl,cf2,rfef,fs} for surveys connecting rectangular $\BMO{}{}$ to the topic of the product Hardy space and its dual, known as product $\BMO{}{}$, which will not be considered in this paper.

Considering shapes in a basis $\cS$ that exhibit a product structure like what was investigated in \cite{dg} naturally leads to a definition of rectangular $\BMO{}{}$ with respect to $\cS$. As will be shown, this product structure can also be exploited to define a rectangular $\BLO{}{}$ space, which can easily be defined in even a multiparameter setting. The relationship between rectangular $\BLO{}{}$ and rectangular $\BMO{}{}$ will be shown to mirror, in some ways, the relationship between $\BLO{}{}$ and $\BMO{}{}$. 

The boundedness of $M$ on $\BMO{}{}(\R^n)$ was first considered by Bennett-DeVore-Sharpley (\cite{bdvs}). They showed that if $Mf\not\equiv\infty$, then $Mf\in \BMO{}{}(\R^n)$ when $f\in \BMO{}{}(\R^n)$. In \cite{ben}, Bennett refined this result, showing that if $Mf\not\equiv\infty$, then $M$ is bounded from $\BMO{}{}(\R^n)$ to $\BLO{}{}(\R^n)$. In fact, he showed the stronger result with $M$ defined by averages of $f$ as opposed to $|f|$. Further work in this direction can be found in \cite{cf,cun,gol,le,ou,saa,zh}.

As the geometric maximal operator $M_{\cS}$ generalises the Hardy-Littlewood maximal operator $M$ and the space $\BMO{\cS}{}(\R^n)$ generalises $\BMO{}{}(\R^n)$, it makes sense to consider the following problem:

\begin{open}
For what bases $\cS$ is the geometric maximal operator $M_{\cS}$ bounded on $\BMO{\cS}{}(\R^n)$?
\end{open}

Although the result of Bennett-DeVore-Sharpley implies that the basis of cubes is one such basis, it is currently unknown whether this holds for the basis of rectangles. 

This problem is the topic of the present paper. The first purpose of the paper is to establish a class of bases for which $M_{\cS}$ is bounded on $\BMO{\cS}{}(\R^n)$. A basis is said to be engulfing if, roughly speaking, one of two intersecting shapes can be expanded to engulf the other without having to grow too large. This class includes the basis of cubes but excludes the basis of rectangles. It is shown, under an assumption on the basis $\cS$, that (see Theorem \ref{benn}):
\begin{introtheorem}[Engulfing bases]
If $\cS$ is an engulfing basis, then $M_{\cS}$ is bounded from $\BMO{\cS}{}(\R^n)$ to $\BLO{\cS}{}(\R^n)$. 
\end{introtheorem}

As an intermediary step to defining and studying rectangular $\BLO{}{}$ spaces, the product nature of $\BLO{\cS}{}(\R^n)$ is studied in more detail. When the shapes exhibit a certain product structure, it is shown that a product decomposition for $\BLO{\cS}{}(\R^n)$ holds (see Theorem \ref{thispaper}). This is analogous to what was done for $\BMO{\cS}{}(\R^n)$ in \cite{dg}.

The third purpose of this paper is to address the situation when $\cS$ does not possess an engulfing property but is instead a product basis. By this we mean that the shapes in $\cS$ exhibit some product structure with respect to lower-dimensional shapes coming from bases that {\emph {do}} have engulfing. Purely using this product structure, the following theorem is shown in section 6, under certain assumptions on the basis $\cS$ (see Theorem \ref{theorem} for the exact statement):
\begin{introtheorem}[Product bases]
If $\cS$ is a strong product basis, then $M_{\cS}$ is bounded from $\BMO{\cS}{}(\R^n)$ to rectangular $\BLO{\cS}{}(\R^{n_1}\times\R^{n_2}\times\cdots\times\R^{n_k})$, where $n_1+n_2+\ldots+n_k=n$. 
\end{introtheorem}
In particular, this theorem applies to the basis of rectangles, and so it follows that the strong maximal function $M_s$ takes functions from strong $\BMO{}{}$ to rectangular $\BLO{}{}$. 


\section{{\bf Preliminaries}}


Consider $\R^n$ with the Euclidean topology and Lebesgue measure. We call a shape in $\R^n$ any open set $S$ such that $0<|S|<\infty$. By a basis of shapes we mean a collection $\cS$ of shapes $S$ that forms a cover of $\R^n$. Unless otherwise stated, $1\leq p<\infty$. 

Common examples of bases are the collections of all Euclidean balls, $\cB$, all cubes, $\cQ$, and all rectangles, $\cR$. In one dimension, these three choices degenerate to the collection of all (finite) open intervals, $\cI$. Other examples of bases are the collection of all ellipses and balls coming from $p$-norms on $\R^n$. 

Fix a basis of shapes $\cS$. We assume here and throughout the paper that $f$ is a measurable function satisfying $f\in\Lone(S)$ for all shapes $S\in\cS$. This implies that $f$ is locally integrable. 

\begin{definition}
The maximal function of $f$ with respect to the basis $\cS$ is defined as
$$
M_{\cS}f(x)=\sup_{\cS\ni{S}\ni{x}}\dashint_{S}\!|f|.
$$
\end{definition}

Since shapes are open, it follows that $M_{\cS}f$ is lower semicontinuous, hence measurable.
One shows this in much the same way as one shows the lower semicontinuity of the Hardy-Littlewood maximal function.

An important feature of a basis is the question of the boundedness of the corresponding maximal operator on $\Lp$ for $1<p<\infty$. Indeed, there exist bases for which no such $p$ exists: the basis of all rectangles, not necessarily having sides parallel to the coordinate axes (\cite{duo}). 

In \cite{dg}, the space of functions of bounded mean oscillation with respect to a general basis $\cS$ was introduced: 

\begin{definition}
We say that $f$ belongs to $\BMO{\cS}{p}(\R^n)$ if
$$
\|f\|_{\BMO{\cS}{p}}:=\sup_{S\in\cS}\left(\dashint_{S}\!|f-f_S|^p\right)^{1/p}<\infty.
$$
\end{definition}

The notation $\BMO{\cS}{}(\Rn)$ will be reserved for the case where $p=1$. By Jensen's inequality, $\BMO{\cS}{p}(\R^n)\subset \BMO{\cS}{}(\R^n)$ for any $1<p<\infty$ with $\|f\|_{\BMO{\cS}{}}\leq \|f\|_{\BMO{\cS}{p}}$. If the opposite inclusion holds, that is $\BMO{\cS}{}(\R^n)\subset \BMO{\cS}{p}(\R^n)$ for some $1<p<\infty$ with $\|f\|_{\BMO{\cS}{p}}\leq c\|f\|_{\BMO{\cS}{}}$ for some constant $c>0$, then we write $\BMO{\cS}{p}(\R^n)\cong\BMO{\cS}{}(\R^n)$. This holds, in fact for all $1<p<\infty$, if the John-Nirenberg inequality is valid for every $f\in\BMO{\cS}{}(\R^n)$ with uniform constants (see \cite{dg}). This is the case for the basis $\cQ$, for instance, as well as the basis $\cR$ (\cite{kor}).  

There do exist bases that fail to satisfy $\BMO{\cS}{}(\R^n)\subset\BMO{\cS}{p}(\R^n)$ for any $p$. An example is the basis $\cQ_c$ of cubes centred at the origin with sides parallel to the axes (\cite{ly}). 

Note that the maximal function of an $f$ in $\BMO{\cS}{}(\R^n)$ need not be finite almost everywhere. For example, $M_{\cQ}f\equiv\infty$ if $f(z)=-\log|z|\in\BMO{\cQ}{}(\R^n)$.

Many familiar $\BMO{}{}$ properties were shown in \cite{dg} to hold at this level of generality, even when working with functions defined on a domain in $\R^n$. In particular, $\BMO{\cS}{p}$ is a Banach space modulo constants. Moreover, $\BMO{\cS}{p}$ is a lattice: if $f,g\in\BMO{\cS}{p}$, then $h\in \BMO{\cS}{p}$, where $h$ is either $\max(f,g)$ or $\min(f,g)$. This follows readily from writing $\max(f,g)=\frac{1}{2}(f+g+|f-g|)$ and $\min(f,g)=\frac{1}{2}(f+g-|f-g|)$, because the operator $f\mapsto|f|$ is bounded on $\BMO{\cS}{p}$ and $\BMO{\cS}{p}$ is a linear space. 

An important subset of $\BMO{}{}$ that often arises is the class of functions of bounded lower oscillation. Analogously to what was done in \cite{dg} for $\BMO{}{}$, we define this set with respect to a general basis:

\begin{definition}
We say that $f$ belongs to $\BLO{\cS}{}(\Rn)$ if
$$
\|f\|_{\BLO{\cS}{}}:=\sup_{S\in\cS}\dashint_{S}\![f-\essinf_{S}f]<\infty.
$$
\end{definition}

Note that $\BLO{\cS}{}(\R^n)\subset \BMO{\cS}{}(\R^n)$ because, for any shape $S\in\cS$,
$$
\dashint_{S}\!|f-f_S|\leq 2 \dashint_{S}\!|f-\alpha|
$$
holds for any constant $\alpha$ and so, in particular, for $\alpha=\essinf\limits_{S}f$. Moreover, the inclusion can be strict: the function $f(z)=\log|z|$ is an element of $\BMO{\cQ}{}(\R^n)\setminus \BLO{\cQ}{}(\R^n)$. The function $f(z)=-\log|z|$, however, {\emph {is}} in $\BLO{\cQ}{}(\R^n)$. This example shows that, in general, $\BLO{\cS}{}(\R^n)$ fails to be a linear space. 

As such, the approach used above to argue that $\BMO{\cS}{}(\R^n)$ is a lattice is not immediately applicable to $\BLO{\cS}{}(\R^n)$. The following establishes that $\BLO{\cS}{}(\R^n)$ is an upper semilattice; that is, $\max(f,g)\in \BLO{\cS}{}(\R^n)$ whenever $f,g\in \BLO{\cS}{}(\R^n)$. 

\begin{proposition}\label{blolatt}
For any basis $\cS$, $\BLO{\cS}{}(\R^n)$ is an upper semilattice with 
$$
\|\max(f,g)\|_{\BLO{\cS}{}}\leq\|f\|_{\BLO{\cS}{}}+\|g\|_{\BLO{\cS}{}}.
$$ 
\end{proposition}

\begin{proof}
Let $f,g\in \BLO{\cS}{}(\R^n)$ and fix a shape $S\in\cS$. Writing $h=\max(f,g)$ and considering the set $E=\{z\in S:f(z)\geq g(z) \}$, we have that
\[
\begin{split}
\int_{S}\![h-\essinf_{S}h] &= \int_{E}\![f-\essinf_{S}h] + \int_{S\setminus E}\![g-\essinf_{S}h]\\ &\leq \int_{E}\![f-\essinf_{S}f] + \int_{S\setminus E}\![g-\essinf_{S}g]\\ &\leq \int_{S}\![f-\essinf_{S}f] + \int_{S}\![g-\essinf_{S}g]\\&\leq |S|\left[ \|f\|_{\BLO{\cS}{}}+\|g\|_{\BLO{\cS}{}}\right].
\end{split}
\]
Dividing by $|S|$ and taking a supremum over $S\in\cS$ yields the result.
\end{proof}



\section{{\bf Engulfing bases}}


In this section, we provide a generalisation of Bennett's theorem that the maximal function is bounded from $\BMO{}{}$ to $\BLO{}{}$. What is essentially the same proof as that of Bennett holds for a class of bases. The key property is that $\cS$ is an engulfing basis. 

\begin{definition}\label{engulf}
We say that $\cS$ is an engulfing basis if there exist constants $c_d,c_e>{1}$, that may depend on the dimension $n$, such that
\begin{enumerate}
\item[(i)] to each $S\in\cS$ we can associate a shape $\widetilde{S}\in\cS$ such that $\widetilde{S}\supset{S}$ and $|\widetilde{S}|\leq c_d|S|$;
\item[(ii)] for each $S\in\cS$, with $\widetilde{S}$ chosen as in (i), if $T\in\cS$ is such that $S\cap T\neq\emptyset$ and $\widetilde{S}^c\cap T\neq\emptyset$, where $\widetilde{S}^c$ denotes the complement of $\widetilde{S}$, then there exists an $\overline{T}\in\cS$ such that $\overline{T}\supset \widetilde{S}\cup T$ and $|\overline{T}|\leq c_e|T|$. 
\end{enumerate}
\end{definition}

Note that the choice of engulfing shape $\overline{T}$ depends on $S$, $T$, and the choice of shape $\widetilde{S}$ to associate to $S$.

An example of an engulfing basis is the family of balls in $\R^n$ with respect to a $p$-metric, $1\leq p \leq\infty$. The bases $\cB$ and $\cQ$ are special cases, with $p=2$ and $p=\infty$, respectively. 

More generally, the basis of balls in any doubling metric measure space is an engulfing basis. Denote by $B(z,r)$ a ball with centre $z$ and radius $r>0$. Every ball $B_1=B(z,r)$ has a natural double $\widetilde{B}_1=B(z,2r)$ satisfying $\widetilde{B}_1\supset B_1$ and $|\widetilde{B}_1|\leq c_d|B_1|$ for some $c_d>1$. In $\R^n$, we have $c_d=2^n$. Furthermore, if $B_2=B(w,R)$ satisfies $B_1\cap B_2\neq\emptyset$ and $\widetilde{B}_1^c\cap B_2\neq\emptyset$, then $R>r/2$ and there is a ball $\overline{B}_2$ centred at a point in $B_1\cap B_2$ of radius $\max(2R,3r)\leq 6r$. This ball satisfies $\overline{B}_2\supset \widetilde{B}_1\cup B_2$ and $|\overline{B}_2|\leq c_e|B_2|$ for some $c_e>1$. In $\R^n$, we have $c_e=6^n$.


An example of a basis which does not satisfy an engulfing property is $\cR$. No matter what choice of $\widetilde{R}$ made in (i), there is no $c_e$ for which condition (ii) holds. To see this, consider the case $n=2$, as well as the rectangles $R_1=[0,1]\times[0,H]$ and $R_2=[0,H]\times[0,1]$ for $H>1$. Notice that $R_1\cap R_2\neq\emptyset$ and suppose that $\widetilde{R}_1^c\cap R_2\neq\emptyset$, where $\widetilde{R}_1$ containing $R_1$ has been chosen as in (i). Then, the smallest rectangle containing $\widetilde{R}_1$ and $R_2$ must contain $[0,H]\times[0,H]$. Thus,
$$
\frac{|\overline{R}_2|}{|R_2|}\geq\frac{H^2}{H}\rightarrow\infty
$$
as $H\rightarrow\infty$, and so there can be no $c_e<\infty$ satisfying condition (ii) uniformly for all rectangles. 


Now we come to the statement of the theorem.

\begin{theorem}\label{benn}
Let $\cS$ be an engulfing basis such that there exists a $p\in(1,\infty)$ for which $M_{\cS}$ is bounded on $L^p(\R^n)$ with norm $A_p$. If $f\in\BMO{\cS}{p}(\R^n)$, then
\begin{equation}\label{bennett}
\dashint_{S}\!M_{\cS}f\leq c\|f\|_{\BMO{\cS}{p}}+\essinf_{S}M_{\cS}f
\end{equation}
for all $S\in\cS$, where $c$ is a constant depending on $p,n,c_d,c_e,$ and $A_p$. Assuming the right-hand side of \eqref{bennett} is finite for every shape $S\in\cS$, it follows that $M_{\cS}f$ is finite almost everywhere and $M_{\cS}f\in\BLO{\cS}{}(\R^n)$ with 
$$
\|M_{\cS}f\|_{\BLO{\cS}{}}\leq c\|f\|_{\BMO{\cS}{p}}.
$$ 

Moreover, if $\BMO{\cS}{p}(\R^n)\cong\BMO{\cS}{}(\R^n)$, then $\|M_{\cS}f\|_{\BLO{\cS}{}}\leq C\|f\|_{\BMO{\cS}{}}$ holds for all $f\in\BMO{\cS}{}(\R^n)$ for which $M_{\cS}f$ is finite almost everywhere.
\end{theorem}

\begin{remark}
This theorem contains not only that of Bennett, but also the corresponding result of Guzm\'{a}n-Partida (\cite{gp}) for the basis $\cQ_c$. This is an engulfing basis and the boundedness of $M_{\cQ_c}$ on $\Lp$ follows from the fact that $\cQ_c\subset\cQ$ and the boundedness of $M_{\cQ}$ on $\Lp$. \end{remark}

\begin{proof}
Fix $f\in\BMO{\cS}{p}(\R^n)$ and $S\in\cS$. Write $f=g+h$, where $g=(f-f_{\widetilde{S}})\chi_{\widetilde{S}}$ and $h=f_{\widetilde{S}}\chi_{\widetilde{S}}+f\chi_{\widetilde{S}^c}$. Then, by the boundedness of $M_{\cS}$ on $\Lp(\R^n)$, 
$$
\dashint_{S}\!M_{\cS}g\leq  \frac{1}{|S|^{1/p}}\|M_{\cS}g\|_{L^p}\leq  \frac{A_p}{|S|^{1/p}}\|g\|_{L^p}\leq A_pc_d^{1/p}\left(\dashint_{\widetilde{S}}|f-f_{\widetilde{S}}|^p\right)^{1/p}\!\!.
$$
Thus,
\begin{equation}
\label{gben}
\dashint_{S}\!M_{\cS}g \leq A_pc_d^{1/p} \|f\|_{\BMO{\cS}{p}}.
\end{equation}

Fix a point $z_0 \in S$ and a shape $T\in \cS$ such that $T\ni z_0$. If $T\subset \widetilde{S}$, then
$$
\dashint_{T}\!|h|=|f_{\widetilde{S}}|\leq\dashint_{\widetilde{S}}\!|f|\leq M_{\cS}f(z)
$$
for every $z\in \widetilde{S}$. In particular, this is true for every $z\in{S}$, and so 
\begin{equation}\label{small}
\dashint_{T}\!|h|\leq\essinf_{S}M_{\cS}f.
\end{equation}
If $T\cap\widetilde{S}^c\neq\emptyset$, then by the engulfing property there exists a shape $\overline{T}$ containing $T$ and $\widetilde{S}$ such that $|\overline{T}|\leq c_e|T|$. Hence, 
\[
\begin{split}
\dashint_{T}\!|h-f_{\overline{T}}|\leq c_e \dashint_{\overline{T}}\!|h-f_{\overline{T}}|&= \frac{c_e}{|\overline{T}|}\left[|\widetilde{S}||f_{\widetilde{S}}-f_{\overline{T}}|+\int_{\overline{T}\cap\widetilde{S}^c}\!|f-f_{\overline{T}}| \right]\\
&\leq \frac{c_e}{|\overline{T}|}\left[\int_{\widetilde{S}}\!|f-f_{\overline{T}}|+\int_{\overline{T}\cap\widetilde{S}^c}\!|f-f_{\overline{T}}| \right]\\&=c_e\dashint_{\overline{T}}\!|f-f_{\overline{T}}|\leq c_e\left(\dashint_{\overline{T}}\!|f-f_{\overline{T}}|^p\right)^{1/p}\leq c_e\|f\|_{\BMO{\cS}{p}}.
\end{split}
\]
Thus,
$$
\dashint_{T}\!|h|\leq \dashint_{T}\!|h-f_{\overline{T}}|+\dashint_{\overline{T}}\!|f|\leq c_e\|f\|_{\BMO{\cS}{p}} +M_{\cS}f(z)
$$
for every $z\in \overline{T}$. In particular, this is true for every $z\in S$, and so
\begin{equation}\label{large}
\dashint_{T}\!|h|\leq c_e\|f\|_{\BMO{\cS}{p}}+ \essinf_{S}M_{\cS}f.
\end{equation}

Combining \eqref{small} and \eqref{large}, we have the pointwise bound
\begin{equation}
\label{hben}
M_{\cS}h(z_0)\leq c_e\|f\|_{\BMO{\cS}{p}} +\essinf_{S}M_{\cS}f.
\end{equation}

Therefore, combining $\eqref{gben}$ and $\eqref{hben}$, we arrive at
$$
\dashint_{S}\!M_{\cS}f\leq \dashint_{S}\!M_{\cS}g+\dashint_{S}\!M_{\cS}h\leq c\|f\|_{\BMO{\cS}{p}} +\essinf_{S}M_{\cS}f.
$$
\end{proof}


\section{{\bf Product structure}}


In this section, we follow Section 8 of \cite{dg}. Write $\R^n=\R^{n_1}\times\R^{n_2}\times\ldots\times\R^{n_k}$, where $n_1+n_2+\ldots+n_k=n$ and $2\leq{k}\leq{n}$. Let $\cS$ be a basis of shapes in $\R^n$ and, for each $1\leq{i}\leq{k}$, let $\cS_i$ be a basis of shapes in $\R^{n_i}$. For $z\in\R^n$, write $\hat{z}_i$ when the $i$th component (coming from $\R^{n_i}$) has been deleted and define $f_{\hat{z}_i}$ to be the function on $\R^{n_i}$ obtained from $f$ by fixing the other components equal to $\hat{z}_i$.

We can define a $\BMO{}{}$ space on $\R^n$ that measures uniform ``lower-dimensional" bounded mean oscillation with respect to $\cS_i$ in the following way.

\begin{definition}\label{ld}
A function $f\in \Loneloc(\R^n)$ is said to be in $\BMO{\cS_{i}}{p}{(\R^n)}$ if $f_{\hat{z}_i}\in \BMO{\cS_{i}}{p}\!(\R^{n_i})$ uniformly in $\hat{z}_i$; i.e.
$$
\|f\|_{\BMO{\cS_{i}}{p}\!(\R^n)}:=\sup_{\hat{z}_i}\|f_{\hat{z}_i}\|_{\BMO{\cS_{i}}{p}\!(\R^{n_i})}<\infty.
$$
\end{definition}

It turns out that under certain conditions on the relationship between the bases $\{\cS_i\}_{i=1}^{k}$ and the overall basis $\cS$, there is a relationship between $\BMO{\cS_{i}}{}{(\R^n)}$ and $\BMO{\cS}{}{(\R^n)}$. We present the theorem, after a definition, below. 

\begin{definition}
Let $\cS$ be a basis of shapes in $\R^n$ and $\cS_i$ be a basis of shapes for $\R^{n_i}$, $1 \leq i \leq k$, where $n_1+n_2+\ldots+n_k=n$.
\begin{enumerate}
\item We say that $\cS$ satisfies the weak decomposition property with respect to $\{\cS_i\}_{i=1}^{k}$ if for every $S\in\cS$, there exist $S_i\in\cS_i$, $1\leq{i}\leq{k}$, such that $S=S_1\times S_2\times\ldots\times S_k$.

\item If, in addition, for every $\{S_i\}_{i=1}^{k}$, $S_i\in\cS_i$, the set $S_1\times S_2\times\ldots\times S_k\in\cS$, then we say that the basis $\cS$ satisfies the strong decomposition property with respect to $\{\cS_i\}_{i=1}^{k}$.
\end{enumerate}
\end{definition}

Starting with bases $\cS_i$ in $\R^{n_i}$, $1\leq{i}\leq k$, the Cartesian product $\cS_1\times \cS_2\times\ldots\times\cS_k$ is a basis of shapes in $\R^n$ with the strong decomposition property with respect to $\{\cS_i\}_{i=1}^{k}$. 

When $\cS_i=\cR_i$, where $\cR_i$ denotes the basis of rectangles in $\R^{n_i}$, the Cartesian product above coincides with the basis $\cR$ in $\R^n$. As such, $\cR$ satisfies the strong decomposition property with respect to $\{\cR_i\}_{i=1}^{k}$. In particular, when $k=n$ and so $n_i=1$ for every $1\leq{i}\leq{n}$, $\cR$ satisfies the strong decomposition property with respect to $\{\cI_i\}_{i=1}^{n}$.

The basis $\cQ$ does not satisfy the strong decomposition property, however, with respect to $\{\cQ_i\}_{i=1}^{k}$ for any $2\leq{k}\leq{n}$, where $\cQ_i$ denotes the basis of cubes in $\R^{n_i}$, as the product of arbitrary cubes (or intervals) may not necessarily be a cube. Nevertheless, $\cQ$ does satisfy the weak decomposition property with respect to $\{\cQ_i\}_{i=1}^{k}$. 

\begin{theorem}[\cite{dg}]
\label{otherpaper}
Let $\cS$ be a basis of shapes in $\R^n$ and $\cS_i$ be a basis of shapes for $\R^{n_i}$, $1 \leq i \leq k$, where $n_1+n_2+\ldots+n_k=n$.
\begin{enumerate}
\item[a)] Let $f\in \bigcap_{i=1}^{k}{\BMO{\cS_{i}}{p}}(\R^n)$. If $\cS$ satisfies the weak decomposition property with respect to $\{\cS_i\}_{i=1}^{k}$, then $f\in\BMO{{\cS}}{p}{(\R^n)}$ with 
$$
\|f\|_{\BMO{\cS}{p}\!(\R^n)}\leq \sum_{i=1}^{k}\|f\|_{\BMO{\cS_{i}}{p}\!(\R^n)}.
$$
\item[b)] Let $f\in\BMO{{\cS}}{p}{(\R^n)}$. If $\cS$ satisfies the strong decomposition property with respect to $\{\cS_i\}_{i=1}^{k}$ and each $\cS_i$ contains a differentiation basis that differentiates $\Loneloc(\R^{n_i})$, then $f\in\bigcap_{i=1}^{k}{\BMO{\cS_{i}}{p}}(\R^n)$ with 
$$
\max_{1\leq i \leq k}\{\|f\|_{\BMO{\cS_{i}}{p}\!(\R^n)}\}\leq 2^{k-1}\|f\|_{\BMO{\cS}{p}\!(\R^n)}.
$$
When $p = 2$, the constant $2^{k-1}$ can be replaced by $1$.  
\end{enumerate}
\end{theorem}

\begin{remark}
The condition that a basis $\cS$ contains a differentiation basis that differentiates $\Loneloc(\R^{n})$ implies that for any $ f\in \Loneloc(\R^{n})$ and $\eps>0$, for almost every $z$ there exists a shape $S\in\cS$ such that $S\ni z$ and 
$$
\left| \dashint_{S}\!f-f(z) \right|<\eps.
$$
The bases of $\cB$ and $\cQ$ are examples of differentiation bases that differentiate $\Loneloc(\R^{n})$. The basis $\cR$ does not differentiate $\Loneloc(\R^{n})$, but it contains $\cQ$ and so $\cR$ also satisfies the assumptions of this theorem.  
\end{remark}

Just as there are ``lower-dimensional" $\BMO{}{}$ spaces, one may define ``lower-dimensional" $\BLO{}{}$ spaces in an analogous manner.

\begin{definition}
A function $f\in \Loneloc(\R^n)$ is said to be in $\BLO{\cS_i}{}(\R^n)$ if 
$$
\|f\|_{\BLO{\cS_{i}}{}\!(\R^n)}:=\sup_{\hat{z}_i}\|f_{\hat{z}_i}\|_{\BLO{\cS_{i}}{}\!(\R^{n_i})}<\infty.
$$
\end{definition}

It turns out that a $\BLO{}{}$-version of Theorem \ref{otherpaper} is true. The proof follows the same lines as that of Theorem \ref{otherpaper} given in \cite{dg}, but we include it here to illustrate how the nature of $\BLO{}{}$ allows us to attain a better constant in part (b). 

\begin{theorem}
\label{thispaper}
Let $\cS$ be a basis of shapes in $\R^n$ and $\cS_i$ be a basis of shapes for $\R^{n_i}$, $1 \leq i \leq k$, where $n_1+n_2+\ldots+n_k=n$.
\begin{enumerate}
\item[a)] Let $f\in \bigcap_{i=1}^{k}{\BLO{\cS_{i}}{}}(\R^n)$. If $\cS$ satisfies the weak decomposition property with respect to $\{\cS_i\}_{i=1}^{k}$, then $f\in\BLO{{\cS}}{}{(\R^n)}$ with 
$$
\|f\|_{\BLO{\cS}{}\!(\R^n)}\leq \sum_{i=1}^{k}\|f\|_{\BLO{\cS_{i}}{}\!(\R^n)}.
$$
\item[b)] Let $f\in\BLO{{\cS}}{}{(\R^n)}$. If $\cS$ satisfies the strong decomposition property with respect to $\{\cS_i\}_{i=1}^{k}$ and each $\cS_i$ contains a differentiation basis that differentiates $\Loneloc(\R^{n_i})$, then $f\in\bigcap_{i=1}^{k}{\BLO{\cS_{i}}{}}(\R^n)$ with 
$$
\max_{1\leq i \leq k}\{\|f\|_{\BLO{\cS_{i}}{}\!(\R^n)}\}\leq \|f\|_{\BLO{\cS}{}\!(\R^n)}.
$$
\end{enumerate}
\end{theorem}

\begin{proof}
We begin by proving the case $k=2$, where $\R^{n}=\R^{n_1}\times\R^{n_2}$ for $n_1+n_2=n$. Write $\cS_x$ for the basis in $\R^{n_1}$ and $x$ for points in $\R^{n_1}$; write $\cS_y$ for the basis in $\R^{n_2}$ and $y$ for points in $\R^{n_2}$.

To prove (a), assume that $\cS$ satisfies the weak decomposition property with respect to $\{\cS_x,\cS_y\}$ and let $f\in\BLO{\cS_x}{}(\R^n)\cap\BLO{\cS_y}{}(\R^n)$. Fixing a shape $S\in\cS$, write $S=S_1\times S_2$ where $S_1\in\cS_x$ and $S_2\in\cS_y$. Then,
$$
\dashint_{S_2}\dashint_{S_1}\![f(x,y)-\essinf_{S}f]\,{d}x\,{d}y=\dashint_{S_2}\dashint_{S_1}\![f(x,y)-\essinf_{S_1}f_{y}]\,{d}x\,{d}y+\dashint_{S_2}\![\essinf_{S_1}f_{y}-\essinf_{S}f]\,{d}y.
$$
For the first integral, we estimate
$$
\dashint_{S_2}\dashint_{S_1}\![f(x,y)-\essinf_{S_1}f_{y}]\,{d}x\,{d}y\leq \dashint_{S_2}\!\norm{f_y}{\BLO{\cS_x}{}\!(\R^{n_1})}\,{d}y\leq \norm{f}{\BLO{\cS_x}{}\!(\R^n)}.
$$
For the second integral, fixing $\eps>0$, the set $E$ of $(x,y)\in S_1\times S_2$ with $\essinf\limits_{S}f>f(x,y)-\eps$ has positive measure. Moreover, the set $F$ of $ (x,y)\in S_1\times S_2$ such that $f(x,y)\geq\essinf\limits_{S_1} f_{y}$ and $f(x,y)\geq\essinf\limits_{S_2} f_{x}$ has full measure, and so $|E\cap F|>0$. Then, taking a point $(x_0,y_0)\in E\cap F$,
\[
\begin{split}
\dashint_{S_2}\![\essinf_{S_1}f_{y}-\essinf_{S}f]\,{d}y&\leq \dashint_{S_2}\![f_{y}(x_0)-f(x_0,y_0)+\eps]\,{d}y\\&= \dashint_{S_2}\![f_{x_0}(y)-f(x_0,y_0)]\,{d}y+\eps\\&\leq \dashint_{S_2}\![f_{x_0}(y)-\essinf_{S_2}f_{x_0}]\,{d}y+\eps\\&\leq \|f_{x_0}\|_{\BLO{\cS_y}{}\!(\R^{n_2})}+\eps\leq \|f\|_{\BLO{\cS_y}{}\!(\R^n)}+\eps .
\end{split}
\]
Therefore, letting $\eps\rightarrow{0^+}$, we conclude that $f\in\BLO{{\cS}}{}(\R^n)$ with 
$$
\|f\|_{\BLO{\cS}{}(\R^n)}\leq \|f\|_{\BLO{\cS_x}{}\!(\R^n)}+\|f\|_{\BLO{\cS_y}{}(\R^n)}.
$$

We now come to the proof of (b). Assume that $\cS$ satisfies the strong decomposition property with respect to $\{\cS_x,\cS_y\}$, and that $\cS_x$ and $\cS_y$ each contain a differentiation basis that differentiates $\Loneloc(\R^{n_1})$ and $\Loneloc(\R^{n_2})$, respectively. Let $f\in\BLO{{\cS}}{}(\R^n)$ and fix a shape $S_1\in\cS_x$. Consider 
$$
g(y)=\dashint_{S_1}\![f_y(x)-\essinf_{S_1}f_y]\,{d}x
$$
as a function of $y$. For any $S_2\in\cS_y$, writing $S=S_1\times S_2$, we have $\essinf\limits_{S_1}f_y\geq\essinf\limits_{S}f$ for almost every $y$, and so
$$
\int_{S_2}\!g(y)\,dy \leq |S_2|\dashint_{S}\![f-\essinf_{S}f]\leq |S_2|\|f\|_{\BLO{\cS}{}(\R^n)},
$$
implying that $g\in \Loneloc(\R^{n_2})$. Let $\eps>0$. Since $\cS_y$ contains a differentiation basis, for almost every $y_0\in \R^{n_2}$ there exists a shape $S_2\in\cS_y$ containing $y_0$ such that
$$
\left|\dashint_{S_2}\!g(y)\,dy-g(y_0) \right|<\eps.
$$
Fix such a $y_0$ and an $S_2$, and write $S=S_1\times S_2$. We have that 
\[
\begin{split}
 \dashint_{S_1}\![f_{y_0}(x)-\essinf_{S_1}f_{y_0}]\,{d}x=g(y_0)\leq \eps + \dashint_{S_2}\!g(y)\,dy \leq \eps + \|f\|_{\BLO{\cS}{}(\R^n)}.
\end{split}
\]
Taking $\eps\rightarrow{0^+}$, since $S_1$ is arbitrary this implies that $f_{y_0}\in\BLO{\cS_x}{}(\R^{n_1})$ with 
$$
\|f_{y_0}\|_{\BLO{\cS_x}{}\!(\R^{n_1})}\leq \|f\|_{\BLO{\cS}{}(\R^n)}.
$$
The fact that this is true for almost every $y_0$ implies that $\|f\|_{\BLO{\cS_x}{}\!(\R^n)}\leq \|f\|_{\BLO{\cS}{}(\R^n)}$. Similarly, one can show that $\|f\|_{\BLO{\cS_y}{}(\R^n)}\leq \|f\|_{\BLO{\cS}{}(\R^n)}$. Thus we have that $f\in\BLO{{\cS_x}}{}(\R^n)\cap\BLO{{\cS_y}}{}(\R^n)$ with 
$$
\max\{\|f\|_{\BLO{\cS_x}{}(\R^n)},\|f\|_{\BLO{\cS_y}{}(\R^n)}\}\leq \|f\|_{\BLO{\cS}{}(\R^n)}.
$$

To prove part (a) for $k>2$ factors, we assume it holds for $k-1$ factors. Write $X=\R^{n_1} \times \R^{n_2} \times \ldots \times \R^{n_{k-1}}$, $Y = \R^{n_k}$, and set $\cS_Y = \cS_k$. Write $x$ for the elements of $\R^{n_1} \times \R^{n_2} \times \ldots \times \R^{n_{k-1}}$ and $y$ for the elements of $\R^{n_k}$. Denote by $\hat{x}_i$ the result of deleting $x_i$ from $x\in\R^{n_1}\times\R^{n_2}\times\ldots\times\R^{n_{k-1}}$.

Assume that $\cS$ has the weak decomposition property with respect to $\{\cS_i\}_{i=1}^{k-1}$. As such, we can define the projection of the basis $\cS$ onto $X$, namely
\begin{equation}
\label{cSx}
\cS_X = \{S_1 \times S_2 \times \ldots \times S_{k-1}: S_i \in \cS_i, \exists S_k \in  \cS_k, \prod_{i = 1}^k S_i \in \cS\}.
\end{equation}
This is a basis of shapes on $X$ which, by definition, has the weak decomposition property with respect to $\{\cS_i\}_{i=1}^{k-1}$. Moreover, $\cS$ has the weak decomposition property with respect to $\{\cS_X,\cS_Y\}$. Beginning by applying the proven case of $k = 2$, we have
$$
\|f\|_{\BLO{\cS}{}(\R^n)}\leq \|f\|_{\BLO{\cS_X}{}\!(\R^n)}+\|f\|_{\BLO{\cS_Y}{}\!(\R^n)}.
$$
Then, we apply the case of $k-1$ factors to $X$ to yield
\[
\begin{split}
\|f\|_{\BLO{\cS_X}{}\!(\R^n)}&= \sup_{y\in Y}\|f_y\|_{\BLO{\cS_{X}}{}\!(X)}\leq \sup_{y\in Y}\sum_{i=1}^{k-1}\|f_y\|_{\BLO{\cS_{i}}{}\!(X)} = \sup_{y\in Y}\sum_{i=1}^{k-1}\sup_{\hat{x}_i}\|({f_y})_{\hat{x}_i}\|_{\BLO{\cS_i}{}\!(\R^{n_i})}\\&\leq \sum_{i=1}^{k-1}\sup_{(\hat{x}_i,y)}\|f_{(\hat{x}_i,y)}\|_{\BLO{\cS_i}{}\!(\R^{n_i})}=\sum_{i=1}^{k-1}\|f\|_{\BLO{\cS_{i}}{}\!(\R^n)}.
\end{split}
\]
Therefore,
$$
\|f\|_{\BLO{\cS}{}(\R^n)}\leq\sum_{i=1}^{k-1}\|f\|_{\BLO{\cS_{i}}{}\!(\R^n)}+\|f\|_{\BLO{\cS_k}{}\!(\R^n)}=\sum_{i=1}^{k}\|f\|_{\BLO{\cS_{i}}{}\!(\R^n)}.
$$

To prove part (b) for $k > 2$ factors, first note that if $\cS$ has the strong decomposition property, then so does $\cS_X$ defined by \eqref{cSx}. We repeat the first part of the proof of (b) for the case $k = 2$ above to reach
$$
\norm{f_{y_0}}{\BLO{\cS_X }{}\!(X)}\leq \norm{f}{\BLO{\cS}{}(\R^n)}
$$
for some $y_0 \in \R^{n_k}$. Now we repeat the process for the function $f_{y_0}$ instead of $f$, with $X_1 = \R^{n_1} \times \R^{n_2} \times \ldots \times \R^{n_{k-2}}$ and $Y_1 = \R^{n_{k-1}}$. This gives
$$
\|(f_{y_0})_{y_1}\|_{\BLO{\cS_{X_1}}{}\!(X_1) }\leq \|f_{y_0}\|_{\BLO{\cS_X}{}\!(X)} \leq \|f\|_{\BLO{ \cS}{}(\R^n)} \quad \forall y_1 \in \R^{n_{k-1}}, y_0 \in \R^{n_k}.
$$ 
We continue until we get to $X_{k-1} = \R^{n_1}$, for which $\cS_{X_k} = \cS_1$, yielding the estimate
$$
\|f_{({y_{k-2}, \ldots, y_0)}}\|_{\BLO{\cS_1}{}\!(\R^{n_1})} \leq \ldots \leq \|f_{y_0}\|_{\BLO{\cS_X}{}\!(X)}
\leq \|f\|_{\BLO{\cS}{}(\R^{n})}
$$
for all $(k-1)$-tuples $y = (y_{k-2}, \ldots, y_0) \in \R^{n_2} \times \ldots \times \R^{n_k}$. Taking the supremum over all such $y$, we have that $f \in \BLO{\cS_{1}}{}(\R^n)$ with
$$
\|f\|_{\BLO{\cS_1}{}\!(\R^n)}=\sup_{y}\|f_y\|_{\BLO{\cS_1}{}\!(\R^{n_1})} \leq \|f\|_{\BLO{\cS}{}(\R^n)}.
$$
A similar process for $i = 2, \ldots, k$ shows that $f \in \BLO{\cS_{i}}{}(\R^n)$ with
$$
\|f\|_{\BLO{\cS_i}{}\!(\R^n)} \leq \|f\|_{\BLO{\cS}{}(\R^{n})}.
$$
\end{proof}


\section{{\bf Rectangular bounded mean oscillation}}

Let $\cS$ be a basis of shapes in $\R^n$ and denote by $\cS_x,\cS_y$ bases of shapes in $\R^{n_1}$ and $\R^{n_2}$, respectively, where $n_1+n_2=n$. Additionally, we maintain the convention that $\cS$ has the strong decomposition property with respect to $\{\cS_x,\cS_y\}$. Writing $x$ for the coordinates in $\R^{n_1}$ and $y$ for those in $\R^{n_2}$, denote by $f_x$ the function obtained from $f$ by fixing $x$. Similarly, $f_y$ is the function obtained from $f$ by fixing $y$. 

We begin by defining the rectangular $\BMO{}{}$ space at this level of generality. 

\begin{definition}
We say that $f$ is in $\BMO{\text{rec},\cS}{}(\R^{n_1}\times\R^{n_2})$ if
\begin{equation}
\|f\|_{\BMO{\text{rec},\cS}{}}:=\sup_{S_1\in \cS_x,S_2\in\cS_y}\dashint_{S_1}\dashint_{S_2}\!|f(x,y)-(f_{x})_{S_2}-(f_{y})_{S_1} +f_{S}|\,{d}y\,{d}x<\infty,
\end{equation}
where $S=S_1\times S_2$.
\end{definition}

In the literature, the classical rectangular $\BMO{}{}$ space corresponds to $\cS_x=\cQ_x$ and $\cS_y=\cQ_y$, and so $\cS$ is the subfamily of $\cR$ that can be written as the product of two cubes. In dimension two, this is the same as $\cR$. 

\begin{proposition}\label{int}
If $f\in \BMO{\cS_x}{}(\R^n)\cup \BMO{\cS_y}{}(\R^n)$, then $f\in \BMO{\text{rec},\cS}{}(\R^{n_1}\times\R^{n_2})$ with 
$$
\|f\|_{\BMO{\text{rec},\cS}{}}\leq 2\min(\|f\|_{\BMO{\cS_x}{}},\|f\|_{\BMO{\cS_y}{}} ).
$$
\end{proposition}

\begin{proof}
We have 
$$
\dashint_{S_1}\dashint_{S_2}\!|f(x,y)-(f_{y})_{S_1} |\,{d}y\,{d}x \leq \sup_{y\in S_2}\dashint_{S_1}\!|f_y(x)-(f_{y})_{S_1} |\,{d}x = \norm{f}{\BMO{\cS_x}{}}
$$
and
\[
\begin{split}
\dashint_{S_1}\dashint_{S_2}\!|(f_{x})_{S_2}-f_{S}|\,{d}y\,{d}x=\dashint_{S_1}\!|(f_{x})_{S_2}-f_{S}|\,{d}x &= \dashint_{S_1}\!\left|\dashint_{S_2}\!f_{x}(y)\,{d}y-\dashint_{S_2}\!(f_y)_{S_1}\,{d}y\right|\,{d}x\\&\leq \dashint_{S_2}\dashint_{S_1}\!|f_{y}(x)-(f_y)_{S_1}|\,{d}x\,{d}y \\&\leq \dashint_{S_2}\|f_y\|_{\BMO{\cS_x}{}}  \leq\norm{f}{\BMO{\cS_x}{}}.
\end{split}
\]
Thus, writing
$$
|f(x,y)-(f_{x})_{S_2}-(f_{y})_{S_1} +f_{S}|\leq |f(x,y)-(f_{y})_{S_1} |+ |(f_{x})_{S_2}-f_{S}|,
$$
it follows that
$$
\dashint_{S_1}\dashint_{S_2}\!|f(x,y)-(f_{x})_{S_2}-(f_{y})_{S_1} +f_{S}|\,{d}y\,{d}x\leq 2 \norm{f}{\BMO{\cS_x}{}}.
$$

Similarly, one shows that $\|f\|_{\BMO{\text{rec},\cS}{}}\leq 2 \norm{f}{\BMO{\cS_y}{}}$.
\end{proof}

\begin{proposition}
\label{strong-rec}
If $f\in \BMO{\cS}{}(\R^n)$, then $f\in \BMO{\text{rec},\cS}{}(\R^{n_1}\times\R^{n_2})$ with
$$
\|f\|_{\BMO{\text{rec},\cS}{}}\leq 3 \|f\|_{\BMO{\cS}{}}.
$$ 
\end{proposition}

\begin{proof}
We have 
$$
\dashint_{S_1}\dashint_{S_2}\!|f(x,y)-f_{S}|\,{d}y\,{d}x\leq \|f\|_{\BMO{\cS}{}},
$$
$$
\dashint_{S_1}\dashint_{S_2}\!|(f_x)_{S_2}-f_S|\,{d}y\,{d}x = \dashint_{S_1}\!\left|\dashint_{S_2}\!f_x(y)\,{d}y-f_S\right|\,{d}x\leq \dashint_{S_1}\dashint_{S_2}\!|f(x,y)-f_{S}|\,{d}y\,{d}x\leq \|f\|_{\BMO{\cS}{}},
$$
and, similarly, 
$$
\dashint_{S_1}\dashint_{S_2}\!|(f_y)_{S_1}- f_S|\,{d}y\,{d}x\leq \|f\|_{\BMO{\cS}{}}. 
$$
Thus, writing 
$$
f(x,y)-(f_x)_{S_2} - (f_y)_{S_1} + f_S = [f(x,y)-f_{S}] - [(f_x)_{S_2}-f_S] - [(f_y)_{S_1}- f_S],
$$
it follows that 
$$
\dashint_{S_1}\dashint_{S_2}\!|f(x,y)-(f_{x})_{S_2}-(f_{y})_{S_1} +f_{S}|\,{d}y\,{d}x\leq 3\|f\|_{\BMO{\cS}{}}.  
$$
\end{proof}

\begin{remark}\label{remark}
In the case where $\cS_x,\cS_y$ each contain a differentiation basis that differentiates $\Loneloc(\R^{n_1})$ and $\Loneloc(\R^{n_2})$, respectively, another proof is possible using Theorem \ref{otherpaper} and Proposition \ref{int}. We identify $\BMO{\cS}{}(\R^n)$ with $\BMO{\cS_x}{}(\R^n)\cap \BMO{\cS_y}{}(\R^n)$, so that
$$
\BMO{\cS}{}(\R^n)\subset\BMO{\cS_x}{}(\R^n)\cup \BMO{\cS_y}{}(\R^n)\subset \BMO{\text{rec},\cS}{}(\R^{n_1}\times\R^{n_2}),
$$
with $\|f\|_{\BMO{\text{rec},\cS}{}}\leq 4 \|f\|_{\BMO{\cS}{}}$.
\end{remark}

Unlike $\BMO{\cS}{}(\R^n)$, it turns out that $\BMO{\text{rec},\cS}{}(\R^{n_1}\times\R^{n_2})$ may not be a lattice. As $\BMO{\text{rec},\cS}{}(\R^{n_1}\times\R^{n_2})$ is a linear space, this property is equivalent to being closed under taking absolute values.

\begin{example}
Consider $f(x,y)=x-y$. We have that $f(x,y)-(f_{x})_{S_2}-(f_{y})_{S_1} +f_{S}$ equals
$$
(x-y)-\left(x-\dashint_{S_2}\!y\,{d}y \right)-\left(\dashint_{S_1}\!x\,{d}x - y \right)+\left(\dashint_{S_1}\!x\,{d}x-\dashint_{S_2}\!y\,{d}y \right)=0,
$$
and so it follows that $f\in \BMO{\text{rec},\cS}{}(\R\times\R)$ for any basis $\cS$. 

For the function $h(x,y)=|f(x,y)|=|x-y|$, however, a computation shows that if $S_1=S_2=I_L=[0,L]$ for $L>0$, then
$$
\dashint_{I_L}\dashint_{I_L}\!|h(x,y)-(h_{x})_{I_L}-(h_{y})_{I_L} +h_{I_L\times I_L}|\,{d}x\,{d}y = \frac{2}{L^2}\int_{0}^{L}\int_{0}^{y}\!\left|2y-\frac{x^2+y^2}{L} -\frac{2L}{3} \right|\,{d}x\,{d}y
$$
by symmetry of the integrand with respect to the line $y=x$. As the integral of the expression inside the absolute value is zero on $I_L\times I_L$, it follows that
$$
\frac{2}{L^2}\int_{0}^{L}\int_{0}^{y}\!\left|2y-\frac{x^2+y^2}{L} -\frac{2L}{3} \right|\,{d}x\,{d}y=\frac{4}{L^2}\iint_{R}\!\left[2y-\frac{x^2+y^2}{L} -\frac{2L}{3} \right]{d}x\,{d}y,
$$
where $R$ is the region defined by the conditions $0\leq{x}\leq{y}$, $0\leq{y}\leq L$, $2y\geq\frac{x^2+y^2}{L}+\frac{2L}{3}$. This region corresponds to the intersection of the disc $x^2+(y-L)^2\leq\frac{L^2}{3}$ and the upper triangle of the square $I_L\times I_L$. Converting to polar coordinates relative to this region, one can compute
\[
\begin{split}
\iint_{R}\!\left[2y-\frac{x^2+y^2}{L} -\frac{2L}{3} \right]{d}x\,{d}y &=\frac{1}{L}\iint_{R}\!\left[\frac{L^2}{3}-x^2-(y-L)^2 \right]{d}x\,{d}y\\&=\frac{1}{L} \int_{0}^{\frac{L}{\sqrt{3}}}\int_{0}^{\pi/2}\!\left(\frac{L^2}{3}-r^2\right)r\,{d}\theta\,{d}r =\frac{\pi L^3}{72}.
\end{split}
\]

Therefore, 
$$
\dashint_{I_L}\dashint_{I_L}\!|h(x,y)-(h_{x})_{I_L}-(h_{y})_{I_L} +h_{I_L\times I_L}|\,{d}x\,{d}y = \frac{4}{L^2}\times \frac{\pi L^3}{72} = \frac{\pi L}{18}\rightarrow\infty\,\text{as}\,L\rightarrow\infty,
$$
showing that $h\not\in \BMO{\text{rec},\cR}{}(\R\times\R)$.
\end{example}

Just as we defined rectangular $\BMO{}{}$, there is a possible analogous definition of rectangular $\BLO{}{}$, defined by having bounded averages of the form
$$
\dashint_{S_1}\dashint_{S_2}\!|f(x,y)-\essinf_{S_2}f_{x}-\essinf_{S_1}f_{y} +\essinf_{S}f|\,{d}y\,{d}x.
$$ 
This definition, however, has a few deficiencies. For one, without the absolute values, the integrand is not necessarily non-negative, which is something one would expect from any class labelled as $\BLO{}{}$. Another property of $\BLO{}{}$ that fails with this definition is being an upper semilattice, as exhibited by the following example.

\begin{example} \label{eg}
If $f(x,y)=x$ and $g(x,y)=y$, then, for any shapes $S_1,S_2$,
$$
f(x,y)-\essinf_{S_2}f_{x}-\essinf_{S_1}f_{y} +\essinf_{S}f = x - x - \essinf_{S_1}x + \essinf_{S_1}x=0
$$
for almost every $x\in S_1$ and 
$$
g(x,y)-\essinf_{S_2}g_{x}-\essinf_{S_1}g_{y} +\essinf_{S}g = y - \essinf_{S_2}y - y + \essinf_{S_2}y=0
$$
for almost every $y\in S_2$. 

Considering the function $h(x,y)=\max(x,y)$, however, and $S_1=S_2=I_L=[0,L]$ for $L>0$. We have that
$$
\dashint_{I_L}\dashint_{I_L}\!|h(x,y)-\essinf_{I_L}h_{x}-\essinf_{I_L}h_{y} +\essinf_{I_L\times I_L}h|\,{d}y\,{d}x
$$
equals 
$$
\frac{1}{L^2}\int_{0}^{L}\int_{0}^{L}\!|\max(x,y)-x-y|\,{d}y\,{d}x =\frac{1}{L^2}\int_{0}^{L}\int_{0}^{L}\!\min(x,y)\,{d}y\,{d}x = \frac{1}{L^2} \times \frac{L^3}{3}=\frac{L}{3},
$$
which tends to $\infty$ as $L\rightarrow\infty$.
\end{example}

These deficiencies are rectified if the essential infimum of $f$ over $S_1\times S_2$ is replaced by the minimum of the essential infima of $f_x$ over $S_2$ and $f_y$ over $S_1$:
$$
\dashint_{S_1}\dashint_{S_2}\!|f(x,y)-\essinf_{S_2}f_{x}-\essinf_{S_1}f_{y} +\min\{\essinf_{S_2}f_x,\essinf_{S_y}f_x \} |\,{d}y\,{d}x.
$$
The identity $\max(a,b)+\min(a,b)=a+b$ gives us that this is equal to 
$$
\dashint_{S_1}\dashint_{S_2}\![f(x,y)-\max\{\essinf_{S_2}f_x,\essinf_{S_1}f_y \} ]\,{d}y\,{d}x,
$$
where the integrand is now clearly non-negative almost everywhere. Boundedness of these averages is the definition we choose for rectangular $\BLO{}{}$. 

An additional benefit to this definition is that it can be defined at a higher level of generality. As in Section 4, we decompose $\R^n=\R^{n_1}\times\R^{n_2}\times\ldots\times\R^{n_k}$ for $2\leq{k}\leq{n}$ and let $\cS_i$ be a basis for $\R^{n_i}$ for each $1\leq{i}\leq{k}$. We continue to assume that $\cS$ has a strong decomposition property, but now with respect to $\{\cS_i\}_{i=1}^{k}$. Recall that $\hat{z}_i$ denotes the result of deleting the $i$th component from $z\in\R^n$ and that $f_{\hat{z}_i}$ denotes the function on $\R^{n_i}$ obtained from $f$ by fixing the other components equal to $\hat{z}_i$.

\begin{definition}
We say that $f$ is in $\BLO{\text{rec},\cS}{}(\R^{n_1}\times\R^{n_2}\times\ldots\times\R^{n_k})$ if
\begin{equation}
\|f\|_{\BLO{\text{rec},\cS}{}}:=\sup_{S\in\cS}\dashint_{S}\![f(z)-\max_{1\leq{i}\leq{k}}\{\essinf_{S_i} f_{\hat{z}_i}\}]\,{d}z<\infty,
\end{equation}
where $S=S_1\times S_2\times\ldots \times S_k$. 
\end{definition}

\begin{proposition}
\label{latt}
$\BLO{\text{rec},\cS}{}(\R^{n_1}\times\R^{n_2}\times\ldots\times\R^{n_k})$ is an upper semilattice with 
$$
\|\max(f,g)\|_{\BLO{\text{rec},\cS}{}}\leq \|f\|_{\BLO{\text{rec},\cS}{}}+\|g\|_{\BLO{\text{rec},\cS}{}}.
$$
\end{proposition}

\begin{proof}
The proof is the same as that of Proposition \ref{blolatt}.
\end{proof}

The following generalisation of Example \ref{eg} illustrates Proposition \ref{latt}.

\begin{example}
If $f$ is a function of some variable $z_{i_1}$ alone, that is $f(z)=F(z_{i_1})$ for some function $F$, and $g$ is a function of $z_{i_2}$ alone, that is $g(z)=G(z_{i_2})$ for some function $G$, then, for any shape $S$,
$$
f(z)-\max_{1\leq{i}\leq{k}}\{\essinf_{S_{i}} f_{\hat{z}_i}\} = F(z_{i_1})-\max(F(z_{i_1}),\essinf_{S_{i_1}} F)=0
$$
for almost every $z_{i_1}\in S_{i_1}$ and 
$$
g(z)-\max_{1\leq{i}\leq{k}}\{\essinf_{S_{i}} g_{\hat{z}_i}\} = G(z_{i_2})-\max(G(z_{i_2}),\essinf_{S_{i_2}} G)=0
$$
for almost every $z_{i_2}\in S_{i_2}$. Therefore, $\|f\|_{\BLO{\text{rec},\cS}{}}=\|g\|_{\BLO{\text{rec},\cS}{}}=0$.

Meanwhile, if $h(z)=\max(f(z),g(z))=\max(F(z_{i_1}),G(z_{i_2}))$, then for any shape $S$,
$$
h(z)-\max_{1\leq{i}\leq{k}}\{\essinf_{S_{i}} h_{\hat{z}_i}\} = \max(F(z_{i_1}),G(z_{i_2})) - \max(F(z_{i_1}),G(z_{i_2}))=0,
$$
and so $\|h\|_{\BLO{\text{rec},\cS}{}}=0$.
\end{example}

This example shows that taking functions of one variable and the maximum of two such functions yields examples of zero elements of rectangular $\BLO{}{}$. Other sources of examples come from the following two propositions.

\begin{proposition}\label{hm}
If $f\in \bigcup_{i=1}^{k}\BLO{\cS_i}{}(\R^n)$, then $f\in \BLO{\text{rec},\cS}{}(\R^{n_1}\times\R^{n_2}\times\ldots\times\R^{n_k})$ with 
$$
\|f\|_{\BLO{\text{rec},\cS}{}}\leq \min_{1\leq i \leq k}\{\|f\|_{\BLO{\cS_i}{}}\}.
$$ 
\end{proposition}

\begin{proof}
Write 
\[
\begin{split}
\dashint_{S}\![f(z)-\max_{1\leq{i}\leq{k}}\{\essinf_{S_i} f_{\hat{z}_i}\}]\,{d}z&\leq \dashint_{S}\![f(z)-\essinf_{S_i} f_{\hat{z}_i}]\,{d}z\\&\leq \dashint_{\hat{S}_i}\!\|f_{\hat{z}_i}\|_{\BLO{\cS_i}{}}\,{d}z\leq \|f\|_{\BLO{\cS_i}{}}
\end{split}
\]
for each $1\leq i \leq k$, where $\hat{S}_i$ is the result of deleting $S_i$ from $S$. From this it follows that $\BLO{\cS_i}{}(\R^n)\subset \BLO{\text{rec},\cS}{}(\R^{n_1}\times\R^{n_2}\times\ldots\times\R^{n_k})$ for $1\leq i \leq k$.
\end{proof}

\begin{proposition}
If $f\in \BLO{\cS}{}(\R^n)$, then $f\in \BLO{\text{rec},\cS}{}(\R^{n_1}\times\R^{n_2}\times\cdots\times\R^{n_k})$ with 
$$
\|f\|_{\BLO{\text{rec},\cS}{}}\leq \|f\|_{\BLO{\cS}{}}.
$$
\end{proposition}

\begin{proof}
This follows from the fact
that
$$
\essinf\limits_{S}f\leq \max_{1\leq{i}\leq{k}}\{\essinf_{S_i} f_{\hat{z}_i}\}
$$ 
holds almost everywhere. Therefore,
$$
\dashint_{S}\![f(z)-\max_{1\leq{i}\leq{k}}\{\essinf_{S_i} f_{\hat{z}_i}\}]\,{d}z\leq \dashint_{S}\![f(z)-\essinf_{S}f]\,{d}z\leq \|f\|_{\BLO{\cS}{}}.
$$
\end{proof}

\begin{remark}
In the case where each $\cS_i$ contains a differentiation basis that differentiates $\Loneloc(\R^{n_i})$, another proof is possible using Theorem \ref{thispaper} and Proposition \ref{hm}, by analogy with Remark \ref{remark}.
\end{remark}

One way of generating a function in $\BLO{\cS}{}(\R^n)$ is demonstrated in the following example. This allows us to exhibit a function in $\BLO{\text{rec},\cS}{}$ with non-zero norm.

\begin{example}
Let $g(x)\in \BLO{}{}(\R)$ and then consider $f(x,y)=g(x-y)$. Writing $\cI_x$ for the basis of intervals in the $x$-direction and analogously for $\cI_y$, we have that $f\in \BLO{\cI_x}{}(\R^2)\cap \BLO{\cI_y}{}(\R^2)$. From Theorem \ref{thispaper}, it follows that $f\in \BLO{\cR}{}(\R^2)$. One can check that $\|f\|_{\BLO{\cR}{}}\leq \|g\|_{\BLO{}{}}$.

In particular, $f(x,y)=-\log|x-y|$ is in $\BLO{\cR}{}(\R^2)$ and has non-zero norm. Regarding $\R^2$ as $\R\times\R$ and taking the rectangle $[0,1]\times[1,2]$, one can compute $\|f\|_{\BLO{\text{rec},\cR}{}}\geq 2\log{2}-1$.  
\end{example}


\section{{\bf Strong product bases}}


Write $\R^n=\R^{n_1}\times\R^{n_2}\times\ldots\times\R^{n_k}$ for $2\leq{k}\leq{n}$ where $n_1+n_2+\ldots+n_k=n$. For $z\in\R^n$, denote by $z_i\in\R^{n_i}$ its $i$th coordinate, according to this decomposition. 

Let $\cS$ be a basis for $\R^n$ and $\cS_i$ be a basis for $\R^{n_i}$ for each $1\leq{i}\leq{k}$. Assume that $\cS$ has the strong decomposition property with respect to $\{\cS_i\}_{i=1}^{k}$, that each $\cS_i$ is an engulfing basis with constants $c_d^i$ and $c_e^i$, and that each $\cS_i$ contains a differentiation basis that differentiates $\Loneloc(\R^{n_i})$. We will call such a basis a strong product basis. 

\begin{theorem}\label{theorem}
Let $\cS$ be a strong product basis such that there exists a $p\in(1,\infty)$ for which $M_{\cS}$ is bounded on $L^p(\R^n)$ with norm $A_p$. If $f\in\BMO{\cS}{p}(\R^n)$, then 
\begin{equation}\label{us}
\dashint_{S}\!M_{\cS}f(z)\,{}dz\leq c\,\|f\|_{\BMO{\cS}{p}} +\dashint_{S}\max_{1\leq i \leq k}\left\{\essinf_{S_i}(M_{\cS}f)_{\hat{z}_i}\right\}{d}z,
\end{equation}
for all $S\in\cS$, where $c$ is a constant depending on $p,n,k,A_p$, $\{c_d^i\}_{i=1}^{k}$, $\{c_e^i\}_{i=1}^{k}$. Assuming that the right-hand side of \eqref{us} is finite for every shape $S\in\cS$, it follows that $M_{\cS}f$ is finite almost everywhere and $M_{\cS}f\in \BLO{\text{rec},\cS}{}(\R^{n_1}\times\R^{n_2}\times\cdots\times\R^{n_k})$ with
$$
\|M_{\cS}f\|_{\BLO{\text{rec},\cS}{}}\leq c\,\|f\|_{\BMO{\cS}{p}}.
$$

Moreover, if $\BMO{\cS}{p}(\R^n)\cong\BMO{\cS}{}(\R^n)$, then $\|M_{\cS}f\|_{\BLO{\text{rec},\cS}{}}\leq C\,\|f\|_{\BMO{\cS}{}}$ holds for all $f\in \BMO{\cS}{}(\R^n)$ for which $M_{\cS}f$ is finite almost everywhere.
\end{theorem}


\begin{proof}
Fix $f\in\BMO{\cS}{p}(\R^n)$ and $S\in\cS$. We write $S=S_1\times S_2\times\ldots\times S_k$, where $S_i\in\cS_i$. Here we are using the weak decomposition property of $\cS$. As each $\cS_i$ is an engulfing basis, each $S_i$ has associated to it a shape $\widetilde{S}_i\in\cS_i$ as in Definition \ref{engulf}, and so we write $\widetilde{S}$ for the shape in $\cS$ formed by $\widetilde{S}_1\times \widetilde{S}_2\times\ldots\times \widetilde{S}_k$. Here we are using the strong decomposition property of $\cS$. 

For $I\subset\{1,2,\ldots,k\}$, we denote by $I^c$ the set $\{1,2,\ldots,k\}\setminus I$. For a fixed shape $S\in\cS$ and $I\subset \{1,2,\ldots,k\}$, consider the family of shapes
\begin{equation}
\cF_I(S)=\{T\in\cS: T\cap{S}\neq\emptyset\;\text{and}\; T_i\cap \widetilde{S}_i^c\neq\emptyset\Leftrightarrow i\in I \}.
\end{equation}
This is the family of shapes that intersect $S$ and ``stick out" of $\widetilde{S}$ in the directions corresponding to $I$. The notation indicating dependence on $S$ may be suppressed when it has been fixed and there is little possibility of confusion.

Let $x$ denote the $I$-coordinates of $z$, that is those coordinates $\{z_i\in\R^{n_i}:i\in I\}$, and $y$ denote the $I^c$-coordinates of $z$, that is $\{z_i\in\R^{n_i}:i\in I^c\}$. When $|I|=1$, in which case $y$ is all coordinates except $z_i\in\R^{n_i}$ for some $1\leq{i}\leq{k}$, we write $y=\hat{z}_i$ as in previous sections.

Consider the basis $\cS_I$ in $X=\prod_{i \in I}\R^{n_i}$ defined by
$$
\cS_I=\prod_{i\in I}\cS_i.
$$
For $f\in\Loneloc(\R^n)$, define
$$
\|f\|_{\BMO{\cS_I}{p}\!(\R^n)}=\sup_y\|f_y \|_{\BMO{\cS_I}{p}\!(X) }.
$$

Applying Theorem \ref{otherpaper} to $\cS_I$ and then to $\cS$ which has the strong decomposition property with respect to $\{\cS_i\}_{i=1}^{k}$, we have
\begin{equation}\label{itos}
\|f\|_{\BMO{\cS_I}{p}\!(\R^n)}\leq\sup_y \sum_{i\in{I}}\|f_y\|_{\BMO{\cS_i}{p}\!(X) }\leq\sum_{i\in{I}}\|f\|_{\BMO{\cS_i}{p}\!(\R^n)} \leq c_k\|f\|_{\BMO{\cS}{p}\!(\R^n)},
\end{equation}
where $c_k=2^{k-1}k$.

Writing
$$
M_{I}f(z)=\sup\left\{\dashint_{T}\!|f|:T\in\cF_I(S) \,\text{and}\, T\ni{z} \right\},
$$ 
we have that 
$$
M_{\cS}f(z)=\max_{I\subset \{1,2,\ldots,k\}}M_{I}f(z)
$$
for $z\in S$. As such, we consider each $M_{I}f$ separately.

\underline{Case $I=\emptyset$ or $I^c=\emptyset$}: Here $\cF_I$ consists of those shapes that do not leave $\widetilde{S}$ in any direction when $I=\emptyset$, and those shapes that leave $\widetilde{S}$ in every direction when $I^c=\emptyset$. These two cases are treated together as the proof proceeds as in the proof of Theorem \ref{benn}. 

Write $f=g_I+h_I$, where $g_I=(f-f_{\widetilde{S}})\chi_{\widetilde{S}}$ and $h_I=f_{\widetilde{S}}\chi_{\widetilde{S}}+f\chi_{\widetilde{S}^c}$. Then, by the boundedness of $M_{\cS}$ on $L^p(\R^n)$, 
$$
\dashint_{S}\!M_{\cS}g_I\leq\frac{1}{|S|^{1/p}}\|M_{\cS}g_I\|_{L^p}\leq \frac{A_p}{|S|^{1/p}}\|g_I\|_{L^p}\leq A_pc_d^{1/p}\left(\dashint_{\widetilde{S}}|f-f_{\widetilde{S}}|^p\right)^{1/p},
$$
where $c_d=c_d^1\times c_d^2\times\cdots\times c_d^k$. Thus,
\begin{equation}
\label{gnot}
\dashint_{S}\!M_{I}g_I\leq \dashint_{S}\!M_{\cS}g_I \leq A_pc_d^{1/p} \|f\|_{\BMO{\cS}{p}}.
\end{equation}

Fix a point $z_0\in S$ and, for the moment, a shape $T\in\cF_{I}$ such that $T\ni z_0$. When $I=\emptyset$, this implies that $T\subset \widetilde{S}$ and so
$$
\dashint_{T}\!|h_I|\leq \dashint_{\widetilde{S}}\!|f|\leq M_{\cS}f(z)
$$
for every $z\in \widetilde{S}$. In particular, this is true for every $z\in S$ and so 
$$
\dashint_{T}\!|h_I|\leq \essinf_{S}M_{\cS}f.
$$
Hence, we have the pointwise bound
\begin{equation}
\label{hnot}
M_{I}h_I(z_0)\leq \essinf_{S}M_{\cS}f.
\end{equation}

When $I^c=\emptyset$, for each $1\leq i\leq k$ there is a shape $\overline{T}_i\in \cS_i$ containing $T_i$ and $\widetilde{S}_i$ such that $|\overline{T}_i|\leq c_e^i|T_i|$. We then create the shape $\overline{T}=\overline{T}_1\times\overline{T}_2\times\ldots\times\overline{T}_k$. This satisfies $\overline{T}\supset T\cup \widetilde{S}$ and $|\overline{T}|\leq c_e |T|$, where $c_e=c_e^1\times c_e^2\times\ldots\times c_e^k$, and so
\[
\begin{split}
\dashint_{T}\!|h_{I}-f_{\overline{T}}|\leq c_e \dashint_{\overline{T}}\!|h_{I}-f_{\overline{T}}|&= \frac{c_e}{|\overline{T}|}\left[|\widetilde{S}||f_{\widetilde{S}}-f_{\overline{T}}|+\int_{\overline{T}\cap\widetilde{S}^c}\!|f-f_{\overline{T}}| \right]\\
&\leq \frac{c_e}{|\overline{T}|}\left[\int_{\widetilde{S}}\!|f-f_{\overline{T}}|+\int_{\overline{T}\cap\widetilde{S}^c}\!|f-f_{\overline{T}}| \right]\\
&=c_e\dashint_{\overline{T}}\!|f-f_{\overline{T}}|\leq c_e \left(\dashint_{\overline{T}}\!|f-f_{\overline{T}}|^p\right)^{1/p}\leq c_e\|f\|_{\BMO{\cS}{p}}.
\end{split}
\]
Hence,
$$
\dashint_{T}\!|h_{I}|\leq \dashint_{T}\!|h_{I}-f_{\overline{T}}|+\dashint_{\overline{T}}\!|f|\leq c_e\|f\|_{\BMO{\cS}{p}} +M_{\cS}f(z)
$$
for every $z\in \overline{T}$, in particular for every $z\in S$, and so
$$
\dashint_{T}\!|h_{I}|\leq c_e\|f\|_{\BMO{\cS}{p}} + \essinf_{S}M_{\cS}f.
$$
Thus, we have the pointwise bound
\begin{equation}
\label{hfull}
M_{I}h_{I}(z_0)\leq c_e\|f\|_{\BMO{\cS}{p}} +\essinf_{S}M_{\cS}f.
\end{equation}

\underline{Case $I\neq\emptyset,I^c\neq\emptyset$}:
Here the shapes in $\cF_I$ leave $\widetilde{S}$ only in those directions corresponding to $I$. Write $S_I$ for $\prod_{i\in I}S_i$ and $\widetilde{S}_I$ for $\prod_{i\in I}\widetilde{S}_i$.

Write $f=g_I+h_I$, where $g_I=(f-(f_y)_{\widetilde{S}_{I}})\chi_{\widetilde{S}}$ and $h_I=(f_y)_{\widetilde{S}_{I}}\chi_{\widetilde{S}}+f\chi_{\widetilde{S}^c}$. Then, by the boundedness of $M_{\cS}$ on $L^p(\R^n)$,
$$
\dashint_{S}\!M_{\cS}g_I\leq \frac{1}{|S|^{1/p}}\|M_{\cS}g_I\|_{L^p}\leq \frac{A_p}{|S|^{1/p}}\|g_I\|_{L^p}=A_pc_d^{1/p}\left(\dashint_{\widetilde{S}}\!|f-(f_y)_{\widetilde{S}_I}|^p\right)^{1/p},
$$
where $c_d=c_d^1\times c_d^2\times\cdots\times c_d^k$. As 
$$
\left(\dashint_{\widetilde{S}}\!|f-(f_y)_{\widetilde{S}_I}|^p\right)^{1/p} = \left(\dashint_{\widetilde{S}_{I^c}}\left( \dashint_{\widetilde{S}_I}\!|f_y(x)-(f_y)_{\widetilde{S}_I}|^p\,{d}x\right){d}y \right)^{1/p}\leq\|f\|_{\BMO{\cS_I}{p}(\R^n)},
$$
we have
\begin{equation}
\label{gone}
\dashint_{S}\!M_{I}g_I\leq \dashint_{S}\!M_{\cS}g_I\leq A_pc_d^{1/p}\|f\|_{\BMO{\cS_I}{p}} \leq A_pc_d^{1/p}c_k\|f\|_{\BMO{\cS}{p}(\R^n)},
\end{equation}
where the last inequality follows from \eqref{itos}.

Fix a point $z_0\in S$ and, for the moment, a shape $T\in\cF_I$ such that $T\ni z_0$. For each $i\in I$, there is a shape $\overline{T}_i\in \cS_i$ containing $T_i$ and $\widetilde{S}_i$ such that $|\overline{T}_i|\leq c_e^i|T_i|$. We then create the shape $\overline{T}_I=\prod_{i\in I} \overline{T}_i$. This satisfies $\overline{T}_I\supset T_I\cup \widetilde{S}_I$ and $|\overline{T}_I|\leq c_e^I |T_I|$, where $c_e^I=\prod_{i\in I} c_e^i$. For $i\notin{I}$, write $\overline{T}_i=T_i$ and recall that $T_i\subset\widetilde{S}_i$. Then, we form the shape $\overline{T}=\overline{T}_1\times \overline{T}_2\times\cdots\times\overline{T}_k$.

Fixing $y\in T_{I^c}\subset \widetilde{S}_{I^c}$, we proceed as in the proof of Theorem \ref{benn}, but work only with the directions in $I$:
\[
\begin{split}
\dashint_{T_I}\!|(h_I)_y(x)-(f_y)_{\overline{T}_I}|\,{d}x&\leq c_e^I \dashint_{\overline{T}_I}\!|(h_I)_y(x)-(f_y)_{\overline{T}_I}|\,{d}x\\&= \frac{c_e^I}{|\overline{T}_I|}\left[\int_{\widetilde{S}_I}\!|(f_y)_{\widetilde{S}_I}-(f_y)_{\overline{T}_I}|\,{d}x+\int_{\overline{T}_I\cap\widetilde{S}_I^c}\!|f_y(x)-(f_y)_{\overline{T}_I}|\,{d}x \right]\\
&\leq \frac{c_e^I}{|\overline{T}_I|}\left[\int_{\widetilde{S}_I}\!|f_y(x)-(f_y)_{\overline{T}_I}|\,{d}x+\int_{\overline{T}_I\cap\widetilde{S}_I^c}\!|f_y(x)-(f_y)_{\overline{T}_I}|\,{d}x \right]\\
&=c_e^I\dashint_{\overline{T}_I}\!|f_y(x)-(f_y)_{\overline{T}_I}|\,{d}x\leq c_e^I \left(\dashint_{\overline{T}_I}\!|f_y(x)-(f_y)_{\overline{T}_I}|^p\,{d}x\right)^{1/p}
\\&\leq c_e^I\|f\|_{\BMO{\cS_I}{p}(\R^n)}\leq c_e^Ic_k\|f\|_{\BMO{\cS}{p}(\R^n)}
\end{split}
\]
by \eqref{itos}. Thus,
$$
\dashint_{T_I}\!|(h_I)_y(x)|\,{d}x\leq \dashint_{T_I}\!|(h_I)_y(x)-(f_y)_{\overline{T}_I}|\,{d}x+\dashint_{\overline{T}_I}\!|f_y(x)|\,{d}x\leq c_e^Ic_k\|f\|_{\BMO{\cS}{p}} +\dashint_{\overline{T}_I}\!|f_y(x)|\,{d}x.
$$
From here, integrating over $y\in T_{I^c}$, we have that 
$$
\dashint_{T}\!|h_I|\leq c_e^Ic_k\|f\|_{\BMO{\cS}{p}} + \dashint_{\overline{T}}\!|f|\leq c_e^Ic_k\|f\|_{\BMO{\cS}{p}}+M_{\cS}f(z)
$$
for any $z\in \overline{T}$. This is true, in particular, if the $I^c$ coordinates of $z$ are equal to $y_0$, where $y_0$ denotes the $I^c$-coordinates of $z_0$, and  $x\in{S_I}$. Thus,
\begin{equation}\label{hone}
M_{I}h_I(z_0)\leq c_e^Ic_k\|f\|_{\BMO{\cS}{p}} +\essinf_{S_I}(M_{\cS}f)_{y_0}.
\end{equation}

Combining \eqref{gnot} and \eqref{gone} yields
\begin{equation}
\begin{split}\label{g}
\dashint_{S}\max_{I}M_{I}g_I\leq \sum_{I}\dashint_{S}\!M_{I}g_I&\leq 2A_pc_d^{1/p}\|f\|_{\BMO{\cS}{p}} +
\sum_{I\neq\emptyset,I^c\neq\emptyset}\!A_pc_d^{1/p}c_k\|f\|_{\BMO{\cS}{p}}\\&\leq c\,\|f\|_{\BMO{\cS}{p}}.
\end{split}
\end{equation}
We combine \eqref{hnot}, \eqref{hfull}, and \eqref{hone} to yield
$$
\dashint_{S} \max_{I}M_{I}h_I\leq c\,\|f\|_{\BMO{\cS}{p}}+\dashint_{S}\max\left\{\essinf_{S}M_{\cS}f,\max_{I\neq\emptyset,I^c\neq\emptyset}\{\essinf_{S_I}(M_{\cS}f)_{y}\}\right\}.
$$
Since the infimum can only grow as we fix more variables, the inequality 
$$
\max\left\{\essinf_{S}M_{\cS}f,\max_{ I\neq\emptyset,I^c\neq\emptyset}\{\essinf_{S_I}(M_{\cS}f)_{y}\}\right\}\leq \max_{1\leq{i}\leq k}\{\essinf_{S_i}(M_{\cS}f)_{\hat{z}_i}\},
$$
holds almost everywhere in $S$, and so
\begin{equation}\label{h}
\dashint_{S} \max_{I}M_{I}h_I(z)\,{d}z\leq
c\,\|f\|_{\BMO{\cS}{p}}+\dashint_{S}\max_{1\leq{i}\leq k}\{\essinf_{S_i}(M_{\cS}f)_{\hat{z}_i}\}\,{d}z.
\end{equation}

Therefore, \eqref{g} and \eqref{h} imply that
\[
\begin{split}
\dashint_{S}\!M_{\cS}f(z)\,{d}z=\dashint_{S}\!\max_{I}M_{I}f(z)\,{d}z&\leq \dashint_{S}\!\max_{I}M_{I}(g_I(z)+h_I(z))\,{d}z\\&\leq \dashint_{S}\!\max_{I}M_{I}g_I(z)\,{d}z+\dashint_{S}\!\max_{I}M_{I}h_I(z)\,{d}z\\&\leq c\,\|f\|_{\BMO{\cS}{p}}+\dashint_{S}\max_{1\leq{i}\leq k}\{\essinf_{S_i}(M_{\cS}f)_{\hat{z}_i}\}\,{d}z.
\end{split}
\]
\end{proof}

We end by giving two examples of bases that satisfy the conditions of Theorem \ref{theorem}.

\begin{example}
The first example, which is in many ways the model case and the motivation for studying this problem, is $\cR$. When $k=n$, and so $n_i=1$ for every $1\leq{i}\leq{n}$, the basis $\cR$ has the strong decomposition property with respect to $\{\cI_i\}_{i=1}^{n}$, where $\cI_i$ is the basis of all intervals in $\R$. Each basis $\cI_i$ is both a differentiation basis that differentiates $\Loneloc(\R)$ and an engulfing basis (one can take $c_d=2$ and $c_e=4$). Moreover, the strong maximal function, $M_{s}$, is well known to be bounded on $L^p(\R^n)$ for all $1<p<\infty$ (\cite{jmz}). The anisotropic version of the John-Nirenberg inequality due to Korenovskii (\cite{kor,ko}) implies that $\BMO{\cR}{p}(\R^n)\cong\BMO{\cR}{}(\R^n)$ for all $1<{p}<\infty$. Therefore, $M_{s}$ maps $\BMO{\cR}{}(\R^n)$ to $\BLO{\text{rec},\cR}{}(\R\times\R\times\cdots\times\R)$.
\end{example}

\begin{example}
A second example is when $k=2$. Denote by $\cB_{n-1}$ the basis of all Euclidean balls in $\R^{n-1}$ and by $\cI$ the basis of intervals in $\R$. The differentiation and engulfing properties of these bases are known. In $\R^n=\R^{n-1}\times\R$, define a cylinder to be the product of a ball $B\in \cB_{n-1}$ and an interval $I\in \cI$. The basis of all such cylinders $\cC$ has the strong decomposition property with respect to $\{\cB_{n-1},\cI \}$. 

By comparing (in the sense of Definition 2.2 in \cite{dg}) these shapes to a family of rectangles, the $L^p(\R^n)$ boundedness of $M_{\cC}$ for any $1<p<\infty$ follows from that of $M_{s}$. Moreover, it can be shown along the lines of the work of Korenovskii \cite{kor,ko} that the John-Nirenberg inequality holds for $\cC$, and so $\BMO{\cC}{p}(\R^n)\cong\BMO{\cC}{}(\R^n)$ holds for all $1<p<\infty$. Therefore, $M_{\cC}$ maps $\BMO{\cC}{}(\R^n)$ to $\BLO{\text{rec},\cC}{}(\R^{n-1}\times\R)$.
\end{example}
 
\section*{{\bf Acknowledgements}}

The authors would like to thank Alex Stokolos for bringing to their attention the open problem of the boundedness of the strong maximal function on strong $\BMO{}{}$ and for fruitful discussions. 


\bibliographystyle{amsplain}


\end{document}